\documentclass[a4paper,11pt]{amsart}
\usepackage{color}
\usepackage{cite}
\usepackage{amssymb,amsmath,amscd}
\usepackage[english]{babel}
\usepackage[latin1]{inputenc}

\newtheorem{prop}{Proposition}[section]
\newtheorem{thm}[prop]{Theorem}
\newtheorem{cor}[prop]{Corollary}

\theoremstyle{definition}
\newtheorem{rem}[prop]{Remark}

\theoremstyle{definition}

\def\N{{\mathbb N}}

\newcommand{\calK}{{\mathcal K}}
\newcommand{\calW}{{\mathcal W}}
\newcommand{\calL}{{\mathcal L}}

\newcommand{\calP}{{\mathcal P}}
\newcommand{\calF}{{\mathcal F}}
\newcommand{\calA}{{\mathcal A}}

\newcommand{\calN}{{\mathcal N}}
\newcommand{\calB}{{\mathcal B}}
\newcommand{\calI}{{\mathcal I}}
\newcommand{\calU}{{\mathcal U}}
\newcommand{\calV}{{\mathcal V}}

\newcommand{\bs}[1]{\boldsymbol{#1}}
\newcommand{\pr}{{(p,r)}}

\newcommand{\ra}{\rightarrow}

\newcommand{\e}{\varepsilon}
\newcommand{\xast}{x^{\ast}}

\newcommand{\past}{p^{\ast}}
\newcommand{\rast}{r^{\ast}}

\newcommand{\Xast}{X^{\ast}}

\newcommand{\norm}[1]{\left\lVert #1 \right\rVert}

\theoremstyle{definition}

\renewcommand{\theenumi}{(\alph{enumi})}
\renewcommand{\labelenumi}{\theenumi}

\begin{document}

\title[On $(p,r)$-null sequences and their relatives]{On $(p,r)$-null sequences and their relatives}

\author{Kati Ain}
\address{Kati Ain, Faculty of Mathematics and Computer Science, Tartu University, J. Liivi 2, 50409 Tartu, Estonia}
\email{kati.ain@ut.ee}
\author{Eve Oja}
\address{Eve Oja, Faculty of Mathematics and Computer Science, Tartu University, J. Liivi 2, 50409 Tartu, Estonia; Estonian Academy of Sciences, Kohtu 6, 10130 Tallinn, Estonia} 
\email{eve.oja@ut.ee}
\dedicatory{Dedicated to Professor Albrecht Pietsch on his eightieth birthday}
\thanks{The research was partially supported by
Estonian Science Foundation Grant 8976 and by institutional research funding IUT20-57 of the Estonian Ministry of Education and Research.}

\begin{abstract}
Let $1\leq p < \infty$ and $1\leq r \leq \past$, where $\past$ is the conjugate index of $p$. We prove an omnibus theorem, which provides numerous equivalences for a sequence $(x_n)$ in a Banach space $X$ to be a $\pr$-null sequence. One of them is that $(x_n)$  is $\pr$-null if and only if $(x_n)$ is null and relatively $\pr$-compact. This equivalence is known in the ``limit'' case when $r=\past$, the case of the $p$-null sequence and $p$-compactness. Our approach is more direct and easier than those applied for the proof of the latter result. We apply it also to characterize the unconditional and weak versions of $\pr$-null sequences.
\end{abstract}

\maketitle

\begin{section}{Introduction}
Let $X$ be a Banach space and let $c_0(X)$ denote the space of null sequences in $X$. Recently, Delgado and Pi\~neiro \cite{PD} introduced and studied an interesting class of $p$-null sequences, where $p\geq 1$, which is a linear subspace of $c_0(X)$. In \cite{O-JF}, it was proved that the space of $p$-null sequences in $X$ can be identified with the Chevet--Saphar tensor product $c_0 \hat{\otimes}_{d_p}X$.

On the other hand, there is a strong form of compactness, the $p$-compact\-ness, that has been studied during the last dozen  years in the literature (see, e.g., \cite{ALO, AMR, CK, DOPS, DPS1, GLT, P2, SK1}). The $p$-null sequences can be characterized via the $p$-compactness as follows. (The definitions will be given in Section 2.)

\begin{thm}[Delgado--Pi\~neiro--Oja]\label{DPO}
Let $1\leq p < \infty$. A sequence $(x_n)$ in a Banach space $X$ is $p$-null if and only if $(x_n)$ is null and relatively $p$-compact.
\end{thm}

Theorem \ref{DPO} was discovered in \cite[Proposition 2.6]{PD} and proved in the case of Banach spaces enjoying a version of the approximation property depending on $p$ (by \cite{O-JM}, this version of the approximation property coincides with the classical one for the closed subspaces of $L_p(\mu)$-spaces). For arbitrary Banach spaces, Theorem \ref{DPO} was proved in \cite{O-JF}.

The proof of Theorem \ref{DPO} in \cite{O-JF} relies on the above-mentioned description of the space of $p$-null sequences as a Chevet--Saphar tensor product. Very recently, an alternative natural proof was found by Lassalle and Turco \cite{LT} who rediscovered and applied a powerful theory due to Carl and Stephani \cite{CS} from 1984. Key concepts of the Carl--Stephani theory are $\calA$-null sequences and $\calA$-compact sets in Banach spaces, which are defined for an arbitrary operator ideal $\calA$. Lassalle--Turco's proof in \cite{LT} relies on the following operator ideal version of Theorem \ref{DPO}, deduced from the Carl--Stephani theory in \cite[Proposition 1.4]{LT}.

\begin{thm}[Lassalle--Turco]\label{LT}
Let $\calA$ be an operator ideal. A sequence $(x_n)$ in a Banach space $X$ is $\calA$-null if and only if $(x_n)$ is null and  $\calA$-compact.
\end{thm}

A starting point for the present article was the observation that, in the proof of Theorem \ref{DPO}, Theorem \ref{LT} could be used in a more efficient way than in \cite{LT}. In particular, the technical result \cite[Proposition 1.5]{LT} would not be needed in the proof. Even more, it is obtained for ``free'' as a by-product (see Remark \ref{R3.3}). Moreover, in that way, Theorem \ref{LT} can be applied to prove results similar to Theorem \ref{DPO} also in cases when the method of \cite{O-JF} cannot be applied. One of such cases is, for instance, the one that involves the recent concepts of $(p,r)$-compactness \cite{ALO}  and of $(p,r)$-null sequences \cite{AO1}.

In Section 3, we prove an omnibus theorem, Theorem \ref{omni}, which provides six equivalent properties  for a sequence in a Banach space to be a $\pr$-null sequence. For completeness, let us cite here the part of the omnibus Theorem \ref{omni} which directly corresponds to Theorem \ref{DPO}.

\begin{thm}\label{th1.3}
Let $1 \leq p < \infty$ and $1\leq r \leq \past$, where $\past$ denotes the conjugate index of $p$. A sequence $(x_n)$ in a Banach space $X$ is $\pr$-null if and only if $(x_n)$ is null and relatively $\pr$-compact.
\end{thm}

Let us remark that in the ``limit'' case $r=\past$, the $(p,\past)$-null and $(p,\past)$-compactness are precisely the $p$-null and $p$-compactness. This is, in fact, the only special case when Theorem \ref{th1.3} could be proved by the method in \cite{O-JF}. The reason is simple: the method in \cite{O-JF} uses the Hahn--Banach  theorem. But the $\pr$-context provides a suitable norm only if $r= \past$, and in all other cases merely quasi-norms are available. But, as well known, quasi-normed spaces do not enjoy the Hahn--Banach theorem.

The approach developed in Section 3 is applied in Section 4 to characterize the unconditional and weak versions of $\pr$-null sequences.

Our notation is standard. We consider Banach spaces over the same, either real or complex, field $\mathbb K$. The closed unit ball of a Banach space $X$ is denoted by $B_X$. 

We denote by $\calL$, $\calW$, $\calK$, and $\overline \calF$, respectively, the operator ideals  of bounded, weakly compact, compact, and approximable linear operators. We refer to Pietsch's book \cite{P} and the survey paper \cite{DJP} by Diestel, Jarchow, and Pietsch for the theory of operator ideals. Let us recall here only the definition of the operator ideal $\calA^\mathrm{sur}$, the \emph{surjective hull} of an operator ideal $\calA$ (see \cite[Section~2]{S73} and \cite[4.7.1]{P}). An operator $T\in \calL(Y,X)$ \emph{belongs to} $\calA^\mathrm{sur}(Y,X)$ if $Tq \in \calA(Z,X)$ for some surjection $q\in\calL (Z,Y)$. Obviously, $\calA \subset \calA^\mathrm{sur}$. If $\calA = \calA ^\mathrm{sur}$, then $\calA$ is called \emph{surjective}.

The Banach space of all absolutely $p$-summable sequences in $X$ is denoted by $\ell_p(X)$ and its norm by $\norm{\cdot}_p$. By $\ell_p ^w (X)$ we mean the Banach space of weakly $p$-summable sequences in $X$ with the norm $\norm{\cdot}_p^w$ (see, e.g., \cite[pp. 32--33]{DJT}). If $1\leq p \leq \infty$, then $\past$ denotes the conjugate index of $p$ (i.e., $1/p + 1/\past =1$ with the convention $ 1/ \infty =0$).

To simplify notation, we shall use the symbol $\ell_\infty$ instead of $c_0$ and, more generally, $\ell_\infty (X)$ instead of $c_0(X)$ if $X$ is a Banach space.
\end{section}

\begin{section}{Basic concepts and notation}
{\bf 2.1. The $\pr$-compactness of sets and operators.} Let $X$ be a Banach space. Let $1\leq p \leq \infty$ and $1\leq r \leq \past$. We define the \emph{$(p,r)$-convex hull} of a sequence $(x_k)\in \ell_p(X)$ by 
\[
(p,r) \mbox{-conv}(x_k)=\left\{ \sum_{k=1}^\infty a_k x_k : (a_k)\in B_{\ell_r} \right\}.
\]

As in \cite{ALO}, we say that a subset $K$ of $X$ is \emph{relatively $(p,r)$-compact} if $K\subset (p,r)\mbox{-conv}(x_n)$ for some $(x_n)\in \ell_p(X)$. According to Grothendieck's criterion, the $(\infty, 1)$-compactness coincides with the usual compactness (because $(\infty, 1)\mbox{-conv}(x_n)$ is precisely the closed absolutely convex hull of $(x_n)$). The $(p,1)$-compactness was occasionally considered in the 1980s by Reinov \cite{R1984} and by Bourgain and Reinov \cite{BR} in the study of approximation properties of order $s\leq 1$. The $(p, \past)$-compactness was introduced in 2002 by Sinha and Karn \cite{SK1} under the name of \emph{$p$-compactness}. Remark that the $1$-compactness was considered already in 1973 by Stephani \cite[Section~4]{S73} under the name of nuclearity (of sets) (see also Remark \ref{rem2.3}).

The notion of $p$-null sequences is due to Delgado and Pi\~neiro \cite{PD}. It was extended in \cite{AO1} in a verbatim way as follows. We call a sequence $(x_n)$ in $X$ \emph{$(p,r)$-null} if for every $\e > 0$ there exist $(z_k)\in\e B_{\ell_p(X)}$ and  $N\in \mathbb N$ such that $x_n \in (p,r)\mbox{-conv}(z_k)$ for all $n \geq N$. The \emph{$p$-null} sequences in \cite{PD} are precisely the $(p,\past)$-null sequences.

A useful way  to look at $(p,r)$-convex hulls is the following. It is well known and easy to see that every $(x_k)\in \ell_p(X)$ defines a compact, even approximable, operator $\Phi _{(x_k)}: \ell_r \ra X$ through the equality
\[
\Phi _{(x_k)}(a_k)= \sum _{k=1}^\infty a_k x_k, \; (a_k) \in \ell _r.
\]
Clearly,
\[
(p,r)\mbox{-conv}(x_k)= \Phi_{(x_k)}(B_{\ell_r}).
\]

In \cite{ALO}, $\pr$-compact operators were introduced in an obvious way: a linear operator $T: Y \ra X $ is {\it $(p,r)$-compact} if $T(B_Y)$ is a relatively $(p,r)$-compact subset of $X$. 
Let $\calK_{(p,r)}$ denote the class of all $(p,r)$-compact operators acting between arbitrary Banach spaces. Then $\calK_{(p, \past)}= \calK_p$, the class of \emph{$p$-compact operators in the sense of Sinha--Karn} \cite{SK1}. And $\calK_{(p,1)}$ is the class of  \emph{$p$-compact operators in the Bourgain--Reinov sense} (cf. \cite{BR, R1984}).

Properties of $\calK _p$ were studied in \cite{SK1} and, for instance, in the recent papers \cite{DPS1, DPS2, SK2}. In \cite{ALO}, an alternative approach, which is direct and easier than in these articles, was developed to study the (quasi-Banach) operator ideal structure of $\calK_\pr$, among others, encompassing and clarifying main results on $\calK_p = \calK_{(p, \past)}$. (Remark that in the latter case the same approach was independently developed by Pietsch \cite{P2} yielding an important far-reaching theory of the (Banach) operator ideal $\calK_p$.)

The approach in \cite{ALO} starts as follows. One observes that $\calK_\pr$ is a surjective operator ideal (an easy straightforward verification). Another immediate observation is that 
\[
\Phi_{(x_n)}\in \calN _{(p,1,\rast)}(\ell_r, X),
\]
the space of $(p,1,\rast)$-nuclear operators (for the definition of $\calN_{(t,u,v)}$, see \cite[18.1.1]{P}). But then, by the definition of the surjective hull, the injective associate of $\Phi_{(x_n)}$ belongs to $\calN_{(p,1,\rast)}^{\mathrm{sur}}$. Let us denote it by $\overline\Phi_{(x_n)}$. Observing that any $T \in\calK_\pr (Y,X)$ can be factorized as $T= \overline \Phi _{(x_n)}S$, one easily obtains that 
\[
\calK_{(p,r)}=\calN_{(p,1,\rast )}^{\mathrm{sur}}
\]
as operator ideals (see \cite[Theorem 3.2]{ALO}).

\bigskip

{\bf 2.2. Some classes of bounded sets.} Let us introduce some useful notation which is inspired by \cite{S3}, but seems to be more suggestive than the notation in \cite{S3}.

Let $\bs{b}$ denote the class of all bounded subsets of all Banach spaces, and let $\bs{g}$ be a subclass of $\bs{b}$. Let $X$ be a Banach space. Following \cite[Definition~1.1]{S3}, we denote by $\bs{g}(X)$ the family of subsets of $X$ which are of type $\bs{g}$. For instance, $\bs{b}(X)$ is the family of all bounded subsets of $X$.

We denote by $\bs w$ and $\bs k$, respectively, the classes of all relatively weakly compact and relatively compact subsets of all Banach spaces. It is convenient to denote by $\bs{k}_\pr$ the class of all relatively $\pr$-compact sets in all Banach spaces. In particular, $\bs k= \bs k _{(\infty, 1)}$ and $\bs k _p := \bs k_{(p,\past)}$, the class of all relatively $p$-compact sets.

Let $\calA$ be an operator ideal. Denote by $\calA(\bs{g})$ the subclass of $\bs{b}$, which is given as
\[
\calA (\bs{g})(X)= \{ E \subset X : E\subset T(F) \textrm{ for some } F\in \bs{g}(Y) \textrm{ and } T\in \calA(Y, X) \}
\]
where $X$ is an arbitrary Banach space (in \cite{S3}, the notation $\calA \circ \bs{g}$ is used).

In this notation, Grothendieck's criterion of compactness reads as follows.
\begin{prop}[Grothendieck]\label{G}
One has $\bs k = \overline \calF (\bs b ) = \calK (\bs b)$.
\end{prop} 

\begin{proof}
Let $X$ be a Banach space and let $K \in \bs k (X)$. Grothendieck's criterion gives us a sequence $(x_n)\in c_0(X)$ such that $K\subset \Phi _{(x_n)}(B_{\ell_1})$. Since $\Phi _{(x_n)}\in \overline \calF(\ell_1 , X)$, it is clear that $K$ is of type $\overline \calF (\bs b)$. But $\overline \calF (\bs b) \subset \calK (\bs b)$ because $\overline \calF \subset \calK$. Finally, if $K$ is of type $\calK (\bs b)$, then it is relatively compact. 
\end{proof}

Proposition \ref{G} says, in particular, that $\bs k _{(\infty, 1)}= \calK_{(\infty, 1)}( \bs b)$. Using the definitions of $\bs k _\pr$ and $\calK_\pr$ together with the observation (see Section 2.1) that $\Phi_{(x_n)}$ belongs to the operator ideal $\calN_{(p,1,\rast)}$, the above proof yields also the general case.

\begin{prop}\label{G-gen}
Let $1\leq p \leq \infty$ and $1\leq r \leq \past$. Then $\bs k _\pr = \calN_{(p,1,\rast)}(\bs b)= \calK_\pr (\bs b)$.
\end{prop}

\begin{rem}\label{rem2.3}
Using the notion of ideal system of sets (see \cite{S73}), the equalities $\bs k = \calK(\bs b)$ and $\bs w = \calW (\bs b)$ were observed in \cite{S3}. In the special case $p=1$, $r=\infty$, the left-hand equality $\bs k_1= \bs k_{(1, \infty)}= \calN (\bs b)$ of Proposition \ref{G-gen} was proved in \cite{S73}; here $\calN= \calN_{(1,1,1)}$ denotes, as usual, the operator ideal of (classical) nuclear operators. 
\end{rem}

\bigskip

{\bf 2.3. $\calA$-null sequences and $\calA$-compact sets.} Let us now describe the relevant notions (cf. Theorem \ref{LT}) from the Carl--Stephani theory \cite{CS}, which is based on earlier work by Stephani \cite{S72, S73, S3}.

Let $\calA$ be an operator ideal.

Following \cite[Lemma~1.2]{CS}, a sequence $(x_n)$ in a Banach space $X$ is called \emph{$\calA$-null} if there exist a Banach space $Y$, a null sequence $(y_n)$ in $Y$, and $T\in\calA(Y, X)$ such that $x_n =T y_n$ for all $n \in \N$.

Using the notation of Section 2.2 and following \cite[Theorem 1.2]{CS}, we say (as in \cite{LT}) that a subset $K$ of a Banach space $X$ is \emph{$\calA$-compact} if $K$ is of type $\calA(\bs k)$, i.e. $K\in \calA(\bs k) (X)$. 

Using Proposition \ref{G} and \ref{G-gen} we shall see now that the relatively $\pr$-compact sets, $\calN_{(p,1,\rast)}$-compact sets, and $\calK_\pr$-compact sets are all the same.
\begin{prop}\label{prop3.1}
Let $1\leq p \leq \infty$ and $1\leq r \leq \past$. Then $\bs k _\pr = \calN_{(p,1,\rast)}(\bs k)= \calK_\pr (\bs k)$.
\end{prop}
\begin{proof}
We know that $\calN_{(p,1,\rast)}$ is a minimal operator ideal (see \cite[18.1.4]{P}). This means that $\calN_{(p,1,\rast)}= \overline \calF \circ \calN_{(p,1,\rast)} \circ \overline \calF$ (see \cite[4.8.6]{P}). Hence, using Propositions \ref{G-gen} and \ref{G}, we have
\begin{eqnarray*}
\calK_\pr (\bs k)& \subset & \calK_\pr (\bs b) = \bs k_\pr = \calN_{(p,1,\rast)}(\bs b)= (\overline \calF \circ \calN_{(p,1,\rast)})(\overline \calF (\bs b))\\
&=& \overline \calF \circ \calN_{(p,1,\rast)}(\bs k)\subset \calN_{(p,1,\rast)}(\bs k) \subset \calK_\pr (\bs k).
\end{eqnarray*}
This shows that $\bs k _\pr = \calN_{(p,1,\rast)}(\bs k)= \calK_\pr (\bs k)$.
\end{proof}

\begin{rem}
The second equality in Proposition \ref{prop3.1} also follows from the general Carl--Stephani theory. Indeed, for any operator ideal $\calA$, it is known (see \cite[p.~79]{CS}) that a subset is $\calA$-compact if and only if it is $\calA^ \mathrm{sur}$-compact. And (see Section 2.1) $\calN_{(p,1,\rast)}^\mathrm{sur}=\calK_\pr$.
\end{rem}
\end{section}

\begin{section}{An omnibus characterization of $(p,r)$-null sequences}
Theorem \ref{omni} below is an omnibus theorem, which provides six equivalent properties for a sequence in a Banach space to be a $\pr$-null sequence. One of these properties is to be a uniformly $\pr$-null sequence, which is a natural (formal) strengthening of a $\pr$-null sequence.

Let $1 \leq p < \infty$ and $1\leq r \leq \past$. We call a sequence $(x_n)$ in a Banach space $X$ \emph{uniformly $\pr$-null} if there exists $(z_k) \in B_{\ell_p(X)}$ with the following property: for every $\e >0$ there exists $N \in \N$ such that $x_n \in \e \, (p,r)\textnormal{-conv}(z_k)$ for all $n\geq N$. 

We say that $(x_n)$ is \emph{uniformly $p$-null} if it is uniformly $(p,\past)$-null. The latter property was implicitly used in a result by Lassalle and Turco asserting (in the above terminology) that the $p$-null sequences are always uniformly $p$-null (concerning the proof (and its simple alternative), see Remark \ref{R3.3}).

\begin{thm}\label{omni}
Let $1 \leq p < \infty$ and $1 \leq r \leq \past$. For a sequence $(x_n)$ in a Banach space $X$ the following statements are equivalent:
\begin{enumerate}
\item $(x_n)$ is $(p,r)$-null,
\item $(x_n)$ is null and relatively $(p,r)$-compact,
\item $(x_n)$ is null and $\calN_{(p,1,\rast)}$-compact,
\item $(x_n)$ is null and $\calK_{(p,r)}$-compact,
\item $(x_n)$ is $\calN_{(p,1,\rast)}$-null,
\item $(x_n)$ is $\calK_{(p,r)}$-null,
\item $(x_n)$ is uniformly $(p,r)$-null.
\end{enumerate}
\end{thm}

\begin{proof}
An easy verification of (a)$\Rightarrow $(b) can be found in \cite[Proposition~2]{AO1}. For completeness and easy reference, let us present it here. 

Since  $(x_n)$ is $(p,r)$-null, for every $\e >0$ there are $N \in \N$ and $(z_k)\in \ell_p(X)$, $\Vert (z_k) \Vert _p \leq \e$, such that $x_n =\sum _{k=1}^\infty a_k^n z_k $, where $(a_k ^n)_{k=1}^\infty \in B_{\ell _r}$, for all $n \geq N$. Hence, for all $n\geq N$,
\[
\Vert x_n  \Vert\leq \sum _{k=1}^\infty \Vert a_k^n z_k\Vert \leq \Vert (a_k^n)_k \Vert _{\past} \Vert (z_k)\Vert _p \leq \Vert (a_k^n)_k\Vert _r \Vert (z_k)\Vert _p \leq \e, \; 
\]  
and therefore $x_n\ra 0$.

Since $\{x_N, x_{N+1}, ...\}\subset(p,r)$-conv$(z_k)$ and $(z_k)\in \ell_p(X)$, 
the sequence
\[
y_k=\begin{cases}
x_k& \text{if } k< N, \\
z_{k-N+1}  & \text{if } k \geq N,
\end{cases}
\]
is in $\ell_p(X)$ and $x_n \in (p,r)$-conv$(y_k)$ for all $n \in \N$. This means that $(x_n)$ is relatively $\pr$-compact.

Implications (b)$ \Leftrightarrow $(c)$\Leftrightarrow$(d) are immediate from Proposition \ref{prop3.1}.

Implications (c)$\Leftrightarrow$(e) and (d)$\Leftrightarrow$(f) are immediate from Theorem \ref{LT}.

To prove that (f)$\Rightarrow$(g), let $(x_n)$ be a $\calK_{(p,r)}$-null sequence. Then there are a null sequence $(y_n)$ in a Banach space $Y$ and an operator $T \in \calK_\pr (Y,X)$ such that $x_n= Ty_n$ for all $n\in \N$. The $\pr$-compactness of $T$ gives us a sequence $(w_k) \in \ell_p(X)$ such that $T(B_Y)\subset \pr$-conv$(w_k)$.
Now $(z_k):=\left (\frac{w_k}{\norm{(w_k)}_p }\right )\in B_{\ell_p(X)}$, and let $\e>0$. As $(y_n)$ is null in $Y$, for $\e_0:=\frac{\e}{\norm{(w_k)}_p}$ there exists $N\in \N$ such that $Ty_n \in \e_0 T(B_Y)$ for all $n\geq N$. Hence,
\[
x_n\in \e_0\pr\textnormal{-conv}(w_k) =\e_0 \norm{(w_k)}_p \pr \textnormal{-conv}(z_k)= \e \pr\textnormal{-conv}(z_k)
\]
for all $n \geq N$, as desired.

The implication (g)$\Rightarrow$(a) is clear from the definitions, because if $(z_k)\in B_{\ell_p(X)}$, then $(\e z_k)\in \e B_{\ell_p(X)}$ and $\pr \textnormal{-conv}(\e z_k)= \e \pr\textnormal{-conv}(z_k)$.
\end{proof}

\begin{rem}\label{R3.3}
In the special case when $r=\past$, Theorem \ref{omni} contains Theorem \ref{DPO}, complementing it and providing for it a somewhat easier proof than in \cite{LT}. In fact, the technical Lassalle--Turco result \cite[Proposition 1.5]{LT} (inspired by \cite[Theorem 1]{AMR}) is not needed. Even more, this technical result appears as a simple by-product of our proof: it is precisely the special case of the implication (a)$\Rightarrow$(g) when $r= \past$.
\end{rem}

Let $\calA$ be an operator ideal. Let $K$ be an $\calA$-compact set and let $(x_n)$ be an $\calA$-null sequence. If $\calB$ is a larger operator ideal than $\calA$, i.e. $\calA \subset \calB$, then, by definitions, clearly, $K$ is also $\calB$-compact and $(x_n)$ is $\calB$-null. In \cite[Proposition~4.7]{ALO}, it was proved  that 
\[
\calK_\pr = \calI_{(p,1,\rast)}^\mathrm{sur} \circ \calK,
\]
where $\calI_{(p,1, \rast)}$ is the operator ideal of \emph{$(p,1,\rast)$-integral operators} (for the definition of these general integral operators, see \cite[19.1.1]{P}). This equality enables us to extend characterizations (d) and (f) of $\pr$-null sequences of Theorem \ref{omni} to even more larger operator ideal than $\calK_\pr$, namely to $\calI_{(p,1,\rast)}^\mathrm{sur}$.

\begin{prop}\label{prop3.3}
Let $1 \leq p < \infty$ and $1 \leq r \leq \past$. For a sequence $(x_n)$ in a Banach space $X$ the following statements are equivalent:
\begin{enumerate}
\item $(x_n)$ is $(p,r)$-null,
\item $(x_n)$ is null and $\calI_{(p,1,\rast)}^\mathrm{sur}$-compact,
\item $(x_n)$ is $\calI_{(p,1,\rast)}^\mathrm{sur}$-null.
\end{enumerate}
\end{prop}

\begin{proof}
As was mentioned, $\calK_\pr = \calI _{(p,1, \rast)}^\mathrm{sur}\circ \calK$. Hence, using Propositions \ref{G-gen} and \ref{G}, we have
\[
\bs k _\pr = \calK_\pr (\bs b)= \calI_{(p,1,\rast)}^\mathrm{sur}(\calK(\bs b)) = \calI_{(p,1,\rast)}^\mathrm{sur} (\bs k).
\]
This shows that relatively $\pr$-compact sets are exactly $\calI_{(p,1,\rast)}^\mathrm{sur}$-compact sets. The claim now follows from Theorems \ref{omni} and \ref{LT}.
\end{proof}

Concerning the special case when $r=\past$, i.e., $\rast=p$, by definition, the operator ideal of \emph{right $p$-nuclear operators} $\calN^p= \calN_{(p,1,p)}$ (cf. \cite[18.1.1]{P} and, e.g., \cite[p.~140]{Ry}). Also, let $\calP_p$ denote the operator ideal of \emph{absolutely $p$-summing operators} ($p$-summing operators in \cite{DJT}). It was noted in \cite[p.~157]{ALO} that $\calP_p^\mathrm{dual}= \calI_{(p,1,p)}^\mathrm{sur}$. Therefore we can spell out, from Theorem \ref{omni} and Proposition \ref{prop3.3}, the following omnibus characterization of $p$-null  sequences.

\begin{cor}
Let $1 \leq p < \infty$. For a sequence $(x_n)$ in a Banach space $X$ the following statements are equivalent:
\begin{enumerate}
\item $(x_n)$ is $p$-null,
\item $(x_n)$ is null and relatively $p$-compact,
\item $(x_n)$ is null and $\calN^p$-compact,
\item $(x_n)$ is null and $\calK_p$-compact,
\item $(x_n)$ is null and $\calP_p ^\mathrm{dual}$-compact,
\item $(x_n)$ is $\calN^p$-null,
\item $(x_n)$ is $\calK_p$-null,
\item $(x_n)$ is $\calP_p^\mathrm{dual}$-null,
\item $(x_n)$ is uniformly $p$-null.
\end{enumerate}
\end{cor}

\end{section}

\begin{section}{Unconditionally and weakly $\pr$-null sequences}
{\bf 4.1. Unconditional and weak $\pr$-compactnesses.} The (uniformly) $\pr$-null sequences and $\pr$-compactness in a Banach space $X$ are defined in terms of $\pr$-convex hulls using the space $\ell_p(X)$ of absolutely $p$-summable sequences in $X$. In general, $\pr$-convex hulls can be defined using the space $\ell_p ^w(X)$ of weakly $p$-summable sequences in $X$. This is a pretty old idea, going back at least to the paper \cite[p.~51]{CSa} by Castillo and Sanchez in 1993. In \cite{CSa}, the $(p,\past)$-convex hull of $(x_n)\in \ell_p^w(X)$ was considered under the name of $\past$-convex hull of $(x_n)$. In 2002, Sinha and Karn \cite{SK1} developed some of their theory of $p$-compactness in a more general context of weak $p$-compactness. In \cite{SK1}, also the $(p,\past)$-convex hull of $(x_n)\in \ell_p^w(X)$ was used but under the name of $p$-convex hull of $(x_n)\in \ell_p^w(X)$.

Let $1\leq p <\infty$ and $1 \leq r \leq \past$. In the present Section 4, we shall assume that  the definition of the \emph{$\pr$-convex hull} $\pr$-conv$(x_n)$ (see Section 2.1) is extended to $(x_n)\in \ell_p^w(X)$. In this case, the operator $\Phi _{(x_n)}: \ell_r \ra X$ is also well defined and 
\[
(p,r)\textnormal{-conv}(x_n)=\Phi _{(x_n)}(B_{\ell_r}).
\]
But $\Phi_{(x_n)}$ need not be a compact operator any more (see, e.g., Section 4.3).

``Between'' absolutely $p$-summable sequences $\ell_p(X)$ and weakly $p$-summable sequences $\ell_p^w(X)$, there is the Banach space $\ell_p^u(X)$ of \emph{unconditionally $p$-summable sequences} (see, e.g., \cite[8.2,~8.3]{DF}; we follow \cite{BCFP} in our terminology). The space $\ell_p^u(X)$ is defined as the (closed) subspace of $\ell_p^w(X)$, formed by the $(x_n)\in \ell_p^w(X)$ satisfying $(x_n)= \lim _{N\ra \infty}(x_1, ... , x_N, 0,0,...)$ in $\ell_p^w(X)$. The space $\ell_p^u(X)$ was introduced and thoroughly studied by Fourie and Swart \cite{FS1} in 1979. In particular, it follows from \cite[Theorem~1.4]{FS1} that $\Phi _{(x_n)}$ is compact whenever $(x_n)\in \ell_p^u(X)$. In fact, $\Phi _{(x_n)}: \ell_{\past} \ra X$ is compact if and only if $(x_n)\in \ell_p^u(X)$ (see \cite[Theorem~1.4]{FS1} or, e.g., ~\cite[8.2]{DF}). 

It is rather easy to see that our approach in Sections 2 and 3 goes through if $\ell_p(X)$ is replaced with the larger space $\ell_p^u(X)$. Let us start by fixing the relevant terminology and notation.

We define \emph{relatively unconditionally} (respectively, \emph{weakly}) \emph{$\pr$-compact} sets in $X$ by replacing $\ell_p(X)$ with $\ell_p^u(X)$ (respectively, with $\ell_p^w (X)$) in the definition of relatively $\pr$-compact sets. The classes of corresponding sets in all Banach spaces are denoted, respectively, by $\bs u_\pr$ and $\bs w _\pr$. So that $\bs k _\pr \subset \bs u _\pr \subset \bs w _\pr $ and $\bs u _\pr \subset \bs k$. 

A linear operator $T: Y\ra X$ is \emph{unconditionally} (respectively, \emph{weakly}) \emph{$\pr$-compact} if $T(B_Y)$ is a relatively unconditionally (respectively, weakly) $\pr$-compact subset of $X$. Let $\calU_\pr$ and $\calW_\pr$ denote the classes of all unconditionally and weakly $\pr$-compact operators acting between arbitrary Banach spaces, so that $\calK_\pr \subset \calU_\pr \subset \calW_\pr$ and $\calU_\pr \subset \calK$. It is clear from the definitions that $\bs u_\pr = \calU_\pr (\bs b)$ and $\bs w_\pr =\calW_\pr (\bs b)$. An easy straightforward verification, as in the case of $\calK_\pr$ (cf. \cite[Propositions~2.1~and~2.2]{ALO}), shows that $\calU_\pr$ and $\calW_\pr$ are surjective operator ideals. 

Note that $\calW_{(p,\past)} = \calW_p$, the class of \emph{weakly $p$-compact operators}, studied in \cite{SK1}. Similarly, in all cases, we shall write ``$p$-'' instead of ``$(p,\past)$-'', and speak, for instance, about the operator ideal $\calU_p$ of unconditionally $p$-compact operators.
 
\medskip

{\bf 4.2. Unconditionally $\pr$-null sequences.} We define \emph{(uniformly) unconditionally $\pr$-null} sequences in $X$ by replacing $\ell_p(X)$ with $\ell_p^u(X)$ in the corresponding definitions of $\pr$-null and uniformly $\pr$-null sequences. The definition of the weak versions of these concepts will be given in Section 4.3;  it turns out to be unreasonably restrictive to define the weak versions just by replacing $\ell_p(X)$ with $\ell_p^w(X)$.

Let $(x_n)\in \ell_p^u(X)$. Then (see \cite[Lemma~1.2]{FS1}) $x_n=\delta_n y_n$ for some $(\delta _n)\in c_0$ and $(y_n)\in \ell_p^w(X)$. Since, clearly, 
\[
\Phi _{(x_n)}= \sum_{n=1}^\infty e_n \otimes x_n = \sum _{n=1}^\infty \delta _n e_n \otimes y_n
\]
(where $e_n\in \ell_r ^\ast $ are the unit vectors) and (as well known and easy to verify) $(e_n)\in B_{\ell_r^w (\ell_r ^\ast)}$, we have, by the definition of $(t,u,v)$-nuclear operators \cite[18.1.1]{P},
\[
\Phi _{(x_n)}\in \calN_{(\infty , \past, \rast)}(\ell_r, X).
\]
Similarly, as in Section 2.1, we get that
\[
\calU_\pr = \calN^{\mathrm{sur}}_{(\infty, \past, \rast)}.
\]
This implies that 
\[
\calU_\pr = \calK\circ \calU_\pr \circ \calK .
\]
Indeed, as in the proof of Proposition \ref{prop3.1}, $\calN_{(\infty, \past, \rast)}= \overline \calF \circ \calN_{(\infty, \past, \rast)} \circ \overline \calF$, and therefore 
\[
\calU_\pr =(\overline \calF \circ \calN_{(\infty, \past, \rast)}\circ \overline \calF)^\mathrm{sur}\subset \overline \calF ^\mathrm{sur} \circ \calN_{(\infty, \past, \rast)}^\mathrm{sur} \circ \overline \calF ^\mathrm{sur} = \calK \circ \calU_\pr \circ \calK,
\]
because $\overline \calF ^\mathrm{sur}=\calK$ (see, e.g., \cite[4.7.13]{P}).

Further, similarly to Proposition \ref{G-gen}, we have $\bs u_\pr =\calN_{(\infty, \past, \rast)}(\bs b) = \calU_\pr (\bs b)$, which implies (cf. Proposition \ref{prop3.1} and its proof) that $\bs u_\pr = \calN_{(\infty, \past, \rast)} (\bs k)=\calU_\pr (\bs k)$. Using the above facts and proceeding as in the proof of Theorem \ref{omni}, we come to the omnibus characterization of unconditionally $\pr$-null sequences.

\begin{thm}\label{omni_uncond}
Let $1 \leq p < \infty$ and $1 \leq r \leq \past$. For a sequence $(x_n)$ in a Banach space $X$ the following statements are equivalent:
\begin{enumerate}
\item $(x_n)$ is unconditionally $(p,r)$-null,
\item $(x_n)$ is null and relatively unconditionally $(p,r)$-compact,
\item $(x_n)$ is null and $\calN_{(\infty, \past, \rast)}$-compact,
\item $(x_n)$ is null and $\calU_{(p,r)}$-compact,
\item $(x_n)$ is $\calN_{(\infty, \past, \rast )}$-null,
\item $(x_n)$ is $\calU_{(p,r)}$-null,
\item $(x_n)$ is uniformly unconditionally $(p,r)$-null.
\end{enumerate}
\end{thm}
\begin{proof}
It is mostly the verbatim version of the proof of Theorem \ref{omni}. Only the claim that $(x_n)$ is null whenever $(x_n)$ is unconditionally $\pr$-null (see the implication (a)$\Rightarrow $(b)) needs to be commented (also for an easy reference in Section 4.3 below).

So, let $(x_n)$ be unconditionally $\pr$-null. Then, as in the proof of (a)$\Rightarrow$(b) in Theorem \ref{omni}, for every $\e>0$ there are $N\in \N$ and $(z_k)\in \ell_p^u(X)$, $\norm{(z_k)}_p^w \leq \e$, such that $x_n= \sum _{k=1}^\infty a_k^n z_k$, where $(a_k^n)_{k=1}^\infty\in B_{\ell_r}$, for all $n \geq N$. Hence, 
\[
\norm{x_n}=\sup_{\xast \in B_{\Xast}} \vert \xast (x_n)\vert \leq \sup_{\xast \in B_{\Xast}} \sum _{k=1}^\infty \vert a_k^n  \xast (z_k) \vert \leq \norm{(a_k^n)_k}_r \norm{(z_k)}_p^w \leq \e,
\]
for all $n\geq N$, and therefore $x_n\ra 0$.
\end{proof}

Recall (see \cite[Theorem~2.5]{FS2} or, e.g., \cite[18.3.2]{P}) that $\calN_{(\infty, p, \past)}$ coincides with the operator ideal $\mathrm K_p$ of \emph{classical} $p$-compact operators. Following Fourie and Swart \cite{FS1} or Pietsch \cite[18.3.1~and~18.3.2]{P}, a linear operator $T: Y \ra X$ is called \emph{$p$-compact}, i.e., $T\in \mathrm K_p(Y,X)$, if there exist $A\in \calK(Y, \ell_p)$ and $B\in \calK(\ell_p, X)$ such that $T=BA$. Remark (see \cite{O-JM} and \cite{P2}) that $\calK_p$ and $\mathrm K_p$ are notably different as operator ideals. 

Since $\calU_{\past}= \calU_{(\past, p)}= \calN_{(\infty, p, \past)}^{\mathrm{sur}}$, we get that $ \mathrm K_{p}^{\mathrm{sur}} = \calU_{\past}$ as a description of the surjective hull of $\mathrm K_p$.

Let us spell out, from Theorem \ref{omni_uncond}, an omnibus characterization of unconditionally $p$-null (i.e., $(p,\past)$-null) sequences.

\begin{cor}
Let $1 \leq p < \infty$. For a sequence $(x_n)$ in a Banach space $X$ the following statements are equivalent:
\begin{enumerate}
\item $(x_n)$ is unconditionally $p$-null,
\item $(x_n)$ is null and relatively unconditionally $p$-compact,
\item $(x_n)$ is null and $\mathrm K_{\past}$-compact,
\item $(x_n)$ is null and $\calU_p$-compact,
\item $(x_n)$ is $\mathrm K_{\past}$-null,
\item $(x_n)$ is $\calU_p$-null,
\item $(x_n)$ is uniformly unconditionally $p$-null.
\end{enumerate}
\end{cor}

\medskip
{\bf 4.3. Weakly $\pr$-null sequences and weakly $\calA$-null sequences.} Let $1\leq p <\infty$ and $1\leq r \leq \past$, as before. What about the weakly $\pr$-null sequences? It would be natural to expect that they would form a subclass of weakly null sequences, but not a subclass of null sequences as in the case of $\pr$-null sequences (which might be called also absolutely $\pr$-null sequences) or unconditionally $\pr$-null sequences. This means that we cannot employ the ``verbatim'' definition: replacing $\ell_p(X)$ with $\ell_p^w(X)$. 

Indeed (see the proof of Theorem \ref{omni_uncond}), such a ``weakly'' $\pr$-null sequence would always be a null sequence. And, for instance, looking at $X=\ell_{\past}$, every null sequence $(x_n)$ in $X$ would be uniformly ``weakly'' $(p,\past)$-null, because the unit vector basis $(e_k)$ of $X$ belongs to $B_{\ell_p^w(X)}$ and, since $\Phi_{(e_k)}=I_X$, we have $x_n= \Phi_{(e_k)}x_n \in \norm{x_n} p \textnormal{-conv}(e_k)$.

To motivate a definition for weakly $\pr$-null sequences, let us make the following observation from Theorem \ref{omni}, yielding two more characterizations of $\pr$-null sequences.

\begin{prop}\label{prop4.3}
Let $1 \leq p < \infty$ and $1 \leq r \leq \past$. For a sequence $(x_n)$ in a Banach space $X$ the following statements are equivalent:
\begin{enumerate}
\renewcommand{\theenumi}{\roman{enumi}}
\renewcommand{\labelenumi}{(\theenumi)}
\item $(x_n)$ is $(p,r)$-null,
\item for every $\e >0$ there exist $(z_k)\in \ell_p(X)$ and $N\in \N$ such that $\norm{x_n}\leq \e$ and $x_n\in \pr \textnormal{-conv}(z_k)$ for all $n\geq N$,
\item there exists $(z_k)\in \ell_p(X)$ with the following property: for every $\e>0$ there exists $N\in \N$ such that $\norm{x_n}\leq \e$ and $x_n \in \pr \textnormal{-conv}(z_k)$ for all $n\geq N$.
\end{enumerate}
\end{prop}

\begin{proof}
The implication (i)$\Rightarrow$(ii) is clear from the proof of Theorem \ref{omni}, the first part of (a)$\Rightarrow$(b). 

From (ii), it is clear that $x_n\ra 0$, and also (fixing, e.g., $\e=1$ and looking at the proof of Theorem \ref{omni}, the second part of (a)$\Rightarrow$(b)) that $(x_n)$ is relatively $\pr$-compact. By Theorem \ref{omni}, (b)$\Rightarrow $(a), $(x_n)$ is $\pr$-null, meaning that (ii)$\Rightarrow $(i). By Theorem \ref{omni}, (b)$\Rightarrow$(g), $(x_n)$ is uniformly $\pr$-null. Hence, assuming that $\e\leq 1$, condition (iii) holds (similarly to the implication (i)$\Rightarrow$(ii) above).

Finally, (iii)$\Rightarrow$(ii) is more than obvious, and we saw above that (ii)$\Leftrightarrow $(i).
\end{proof}

Looking at Proposition \ref{prop4.3}, it seems to be natural to make the following definitions.

Let $(x_n)$ be a sequence in a Banach space $X$. We call $(x_n)$ \emph{weakly $\pr$-null} if for every $\xast \in\Xast$ and every $\e>0$ there exist $(z_k)\in \ell_p^w(X)$ and $N\in \N$ such that $\vert \xast(x_n) \vert \leq \e$ and $x_n\in \pr \textnormal{-conv}(z_k)$ for all $n\geq N$. We call $(x_n)$ \emph{uniformly weakly $\pr$-null} if there exists $(z_k)\in \ell_p^w(X)$ with the following property: for every $\xast \in \Xast$ and every $\e>0$ there exists $N\in \N$ such that $\vert \xast (x_n)\vert\leq \e$ and $x_n \in \pr \textnormal{-conv}(z_k)$ for all $n\geq N$.

Let $\calA$ be an operator ideal. In the present context, it would be natural to complement the Carl--Stephani theory with the concepts of weakly $\calA$-null sequences and weakly $\calA$-compact sets as follows. 

We call a sequence $(x_n)$ in a Banach space $X$ \emph{weakly $\calA$-null} if there exist a Banach space $Y$, a weakly null sequence $(y_n)$ in $Y$, and $T\in\calA(Y, X)$ such that $x_n =T y_n$ for all $n \in \N$. We say that a subset $K$ of $X$ is \emph{weakly $\calA$-compact} if $K$ is of type $\calA(\bs w)$, i.e., $K \in \calA(\bs w)(X)$. (Recall that $\bs w$ denotes the class of all relatively weakly compact sets.)

Two basic facts in the Carl--Stephani theory \cite{CS} are that the classes of $\calA$-null and $\calA^\mathrm{sur}$-null sequences coincide, and so also do $\calA$-compact and $\calA^\mathrm{sur}$-compact sets. The ``weak'' versions of these results do not hold. 

Indeed, let $\calV$ denote the operator ideal of \emph{completely continuous} operators, i.e., of operators who take weakly null sequences to null sequences. Then $\calV ^\mathrm{sur}=\calL$ (see, e.g., \cite[4.7.13]{P}). Consequently, the weakly $\calV$-null sequences are (precisely, because null sequences are $\calK$-null, hence $\calV$-null) the null sequences, but the weakly $\calV^\mathrm{sur}$-null sequences are precisely the weakly null sequences. Similarly, the weakly $\calV$-compact sets are precisely relatively compact:
\[
\calV(\bs w)=\calV ( \calW ( \bs b))= (\calV \circ \calW )( \bs b) = \calK(\bs b)= \bs k 
\]
(see Remark \ref{rem2.3} for the equality $\bs w= \calW(\bs b)$ and, e.g., \cite[3.1.3]{P} for the equality $\calV\circ \calW= \calK$). But $\calV ^\mathrm{sur}(\bs w)= \bs w$.

However, for our purposes, the following analogue of the Lassalle--Turco Theorem \ref{LT}, characterizing \emph{weakly} $\calA$-null sequences, will be sufficient.

\begin{prop}\label{prop4.4}
Let $\calA$ be an operator ideal and let $(x_n)$ be a sequence in a Banach space $X$.
\begin{enumerate}
\item If $(x_n)$ is weakly $\calA$-null, then $(x_n)$ is weakly null and weakly $\calA$-compact.
\item If $(x_n)$ is weakly null and weakly $\calA$-compact, then $(x_n)$ is weakly $\calA^\mathrm{sur}$-null.
\end{enumerate}

In particular, if $\calA$ is surjective, then $(x_n)$ is weakly $\calA$-null if and only if $(x_n)$ is weakly null and weakly $\calA$-compact.
\end{prop}

\begin{proof}
(a) We have $x_n= Ty_n$ for some $T \in \calA(Y,X)$ and weakly null sequence $(y_n)$ in $Y$. Hence $(x_n)$ is weakly null. Since $(y_n)$ is relatively weakly compact in $Y$, $(x_n)$ is weakly $\calA$-compact.

(b) We know that $(x_n)\subset T(K)$ for some $T\in \calA(Y,X)$ and weakly compact subset $K$ of $Y$. We may and shall assume that $0 \in K$. Denote by $\overline T$ the injective associate of $T$. Then $T=\overline T q$, where $q: Y\ra Z:=Y/\ker T$ is the quotient mapping, and $\overline T \in \calA^\mathrm{sur}(Z, X)$ (by the definition of $\calA^\mathrm{sur}$). 

If $q(K)$ and $\overline T(q(K))=T(K)$ are endowed with their weak topologies from $Z$ and $X$, respectively, then $\overline T: q(K)\ra T(K)$ is a continuous bijection, hence a homeomorphism. Let $x_n=Tk_n=\overline Tqk_n$ for some $k_n\in K$ and let $z_n=qk_n$. Then $z_n=\overline T ^{-1}x_n \ra \overline T^{-1}(0)=0$ weakly (recall that $0\in K$ and $(x_n)$ is weakly null by the assumption). Since $x_n=\overline T z_n$ for all $n\in \N$, $(x_n)$ is weakly $\calA^\mathrm{sur}$-null.
\end{proof}

We saw (in Sections 2.2, 2.3, 4.1, 4.2) that $\bs k _\pr = \calK_\pr (\bs b)= \calK_\pr (\bs k)$ and, similarly, $\bs u _\pr = \calU_\pr (\bs b) = \calU_\pr (\bs k)$. Also $\bs w_\pr =\calW_\pr (\bs b)$ (see Section 4.1). In general, $\calW_\pr (\bs b) \neq \calW_\pr (\bs k)$. Indeed, as was mentioned in the beginning of Section 4.3, for $X=\ell_{\past}$, one has $\Phi _{(e_k)}=I_X$. Hence, $\calW_p(X,X)=\calL(X,X)$ and therefore $\calW_p(\bs b)(X)=\bs b(X)$, but $\calW_p(\bs k)(X)=\bs k(X)$. We shall need the fact that in many cases $\calW _\pr (\bs b)= \calW_\pr (\bs w)$.

\begin{prop}\label{prop4.5}
Let $1\leq p < \infty$ and $1 < r \leq \past$ with $r< \infty$ if $p=1$. Then 
\[
\calW_\pr = \calW_\pr \circ \calW \text{ and } \bs w_\pr =\calW_\pr (\bs w).
\]
\end{prop}

\begin{proof}
Let $X$ and $Y$ be Banach spaces and $T \in \calW_\pr(Y,X)$. As in  the case of $\calW_p$ in \cite[pp.~20--21]{SK1} and of $\calK_\pr$ (see Section 2.1), we get a natural factorization $T=\overline \Phi _{(x_n)} S$ with $(x_n)\in \ell_p^w(X)$, where $\overline \Phi_{(x_n)}$ is the injective associate of $\Phi_{(x_n)}$ and $S \in \calL(Y,Z)$, where $Z:= \ell_r / \ker \Phi_{(x_n)}$. Since $\Phi_{(x_n)} \in \calW_\pr (\ell_r, X)$, we have $\overline \Phi _{(x_n)}\in \calW_\pr ^\mathrm{sur}(Z,X) = \calW_\pr (Z,X)$, because $\calW_\pr$ is surjective. Since $\ell_r$ is reflexive, also $Z$ is, and therefore $S \in \calW(Y,Z)$. This proves that $\calW_\pr =\calW_\pr \circ \calW$. Now, using this, we have
\[
\bs w_\pr = \calW_\pr (\bs b)= (\calW_\pr \circ \calW)(\bs b)= \calW_\pr (\calW (\bs b))=\calW_\pr (\bs w). \qedhere
\]
\end{proof}

\begin{rem}\label{rem4.6}
We do not know whether Proposition \ref{prop4.5} holds in the ``limit'' case $r=1$, i.e., for $\calW_{(p,1)}$. It does not hold in the other ``limit'' case $p=1$, $r=\infty$, i.e., for $\calW_1= \calW_{(1, \infty)}$. Indeed, as we saw above, $\calW_1(c_0, c_0)= \calL(c_0, c_0)$, and hence 
\[
\bs w _1(c_0)= \calW_1 (\bs b)(c_0)= \bs b (c_0) \neq \bs w (c_0)= \calW_1(\bs w).
\]
In particular, $\calW_{(1,\infty)} \not \subset \calW$. In all other cases $\calW_\pr \subset \calW$. For $r\neq 1$, this is clear from Proposition \ref{prop4.5}. But $\calW_{(p,1)} \subset \calW_\pr$ (by the definition of $\calW_{(p,\cdot)}$, because $B_{\ell_1}\subset B_{\ell_r}$).
\end{rem}

\begin{rem}\label{rem4.7}
In the case $p=1$, $1\leq r \leq \past$, including also the case $p=1$, $r=\infty$ (cf. Remark \ref{rem4.6}), Proposition \ref{prop4.5} holds in a strong form for a large class of Banach spaces $X$. Namely, for $X$ that does not contain $c_0$ isomorphically. In this case (and only in this case), $\ell_1 ^w (X)=\ell_1^u(X)$, by the classical Bessaga--Pe{\l}czy\'{n}ski theorem \cite[Theorem~5]{BP} (see, e.g., \cite[8.3]{DF}). Therefore (see Section 4.2),
\[
\calW_{(1,r)}(Y,X)=\calU_{(1,r)}(Y,X)=(\calK\circ \calU_{(1,r)}\circ \calK)(Y,X)
\]
for all Banach spaces $Y$, and
\[
\bs w_{(1,r)}(X)= \bs u_{(1,r)}(X) =\calU_{(1,r)}(\bs k)(X)= \calN_{(\infty, \infty, \rast)}(\bs k)(X).
\]
\end{rem}

Keeping in mind that the operator ideal $\calW_\pr$ is surjective (see Section 4.1) we come to an omnibus characterization of weakly $\pr$-null sequences.

\begin{thm}\label{thm4.8}
Let $1\leq p < \infty$ and $1 < r \leq \past$ with $r< \infty$ if $p=1$. For a sequence $(x_n)$ in a Banach space $X$ the following statements are equivalent:
\begin{enumerate}
\item $(x_n)$ is weakly $(p,r)$-null,
\item $(x_n)$ is weakly null and relatively weakly $(p,r)$-compact,
\item $(x_n)$ is weakly null and weakly $\calW_\pr$-compact,
\item $(x_n)$ is weakly $\calW_\pr$-null,
\item $(x_n)$ is uniformly weakly $(p,r)$-null.
\end{enumerate}
\end{thm}

\begin{proof}
(a)$\Rightarrow$(b) It is clear from the definition that $x_n\ra 0$ weakly. Also, by the definition, we have (fixing, e.g., $\e=1$) $N\in \N$ and $(z_k)\in \ell_p^w(X)$ such that $\{x_N, x_{N+1}, ...\}\subset(p,r)$-conv$(z_k)$. Continuing verbatim to the proof of Theorem \ref{omni}, the second part of (a)$\Rightarrow $(b), we see that $(x_n)$ is relatively weakly $\pr$-compact.

Implications (b)$\Leftrightarrow $(c) and (c)$\Leftrightarrow$(d) are immediate from Propositions \ref{prop4.5} and \ref{prop4.4}, respectively.

To prove that (d)$\Rightarrow $(e), let $(x_n)$ be a weakly $\calW_{(p,r)}$-null sequence. Then there are a weakly null sequence $(y_n)$ in a Banach space $Y$ and an operator $T \in \calW_\pr (Y,X)$ such that $x_n= Ty_n$ for all $n\in \N$. The weak $\pr$-compactness of $T$ gives us a sequence $(w_k) \in \ell_p^w(X)$ such that $T(B_Y)\subset \pr$-conv$(w_k)$. We also have an $M>0$ such that $\norm{y_n} \leq M$ for all $n\in \N$. Now $(z_k):=(M w_k )\in \ell_p(X)$ and $x_n\in \pr\textnormal{-conv}(z_k)$ for all $n\in \N$. As $(x_n)$ is weakly null in $X$, for every $\xast \in\Xast$ and $\e>0$ there exists $N\in \N$ such that $\vert \xast(x_n)\vert \leq \e$ for all $n\geq N$. Hence, $(x_n)$ is uniformly weakly $\pr$-null.

The implication (e)$\Rightarrow$(a) is clear from the definitions.
\end{proof}

\begin{rem}
As we saw, all implications of Theorem \ref{thm4.8}, except (b)$\Rightarrow $(c), also hold in the ``limit'' cases $r=1$ and $p=1$, $r=\infty$. In the proof, we used that the implication (b)$\Rightarrow $(c) is immediate from Proposition \ref{prop4.5} (see also Remark \ref{rem4.6}). We do not know whether Theorem \ref{thm4.8} holds in these cases. If $p=1$ and $1\leq r \leq \past$, Theorem \ref{thm4.8} holds in a stronger form for those Banach spaces $X$ that do not contain $c_0$ isomorphically. Indeed, by Remark \ref{rem4.7}, in condition (b), ``weakly $(1,r)$-compact'' is the same as ``unconditionally $(1,r)$-compact'' and in condition (c) ``weakly $\calW_{(1,r)}$-compact'' is the same as ``$\calU_{(1,r)}$-compact'' and also the same as ``$\calN_{(\infty, \infty, \rast)}$-compact''. In condition (d), ``weakly $\calW_{(1,r)}$-null'' is the same as ``weakly $\calU_{(1,r)}\circ \calK$-null'', which is the same as ``$\calU_{(1,r)}$-null'', since compact operators take weakly null sequences to null sequences, i.e., $\calK\subset \calV$ (see, e.g., \cite[1.11.4]{P}). This shows that in the special case when $p=1$, $1\leq r \leq \past$, and $X$ does not contain $c_0$ isomorphically, all conditions of Theorem \ref{omni_uncond} are equivalent to the conditions of Theorem \ref{thm4.8}.
\end{rem}
\end{section}

\end{document}